\documentclass[preprint]{elsarticle}

\usepackage{dsfont}
\usepackage[font=small,labelfont=bf]{caption}
\usepackage{mathcomp}
\usepackage{graphicx}
\journal{European Journal of Combinatorics}

\usepackage{amsthm}
\usepackage{amsmath}
\usepackage{amssymb}
\usepackage{epigraph}
\usepackage{graphicx}
\usepackage{amsmath}
\usepackage{subfigure}
\usepackage{xspace}
\usepackage{color, soul}
\usepackage{mathtools}
\usepackage[colorinlistoftodos,prependcaption,textsize=tiny]{todonotes}

\newcommand{\Fraisse}{Fra\"{i}ss\'{e}\xspace} 
\newcommand{\Jarik}{Ne\v set\v ril\xspace}

\newcommand{\N}{\mathbb{N}} 

\newcommand{\age}[1]{\mathrm{Age}(#1)}
\newcommand{\ndo}[1]{\mathrm{End}(#1)}

\newcommand{\trg}{{\normalfont(}$\bigtriangleup${\normalfont )}\xspace}
\newcommand{\ntrg}{{$\neg$\normalfont(}$\bigtriangleup${\normalfont )}\xspace}
\newcommand{\ctrg}{{\normalfont(}$\therefore${\normalfont)}\xspace}

\newcommand{\kk}[1]{\mathcal K(#1)}
\newcommand{\okk}[1]{{\overline{\mathcal K}}(#1)}
\newcommand{\oo}[1]{\sigma(#1)}

\newcommand{\coord}[1]{\mathrm{address}(#1)}

\newcommand{\fin}[1]{[{#1}]^{<\omega}}

\newtheorem{theorem}{Theorem}
\newtheorem{lemma}[theorem]{Lemma}
\newtheorem{corollary}[theorem]{Corollary}
\newtheorem{proposition}[theorem]{Proposition}
\newtheorem{definition}[theorem]{Definition}

\newdefinition{construction}{Construction}
\newdefinition{remark}[theorem]{Remark}
\newdefinition{example}{Example}
\newdefinition{notation}{Notation}

\newtheorem{claim}[theorem]{Claim}

\begin{document}
\begin{frontmatter}

\title{The independence number of HH-homogeneous graphs and a classification of MB-homogeneous graphs}

\author[ia]{Andr\'{e}s Aranda}
\ead{andres.aranda@gmail.com}
\author[iuuk,cas]{David Hartman}
\address[ia]{Institut f\"ur Algebra, Technische Universit\"{a}t Dresden, Zellescher Weg 12-14, Dresden.}
\address[iuuk]{Computer Science Institute of Charles University, Charles University, Malostransk\'e n\'{a}m. 25, Prague 1}
\address[cas]{Institute of Computer Science of the Czech Academy of Sciences, Pod Vod\'{a}renskou v\v{e}\v{z}\'{i} 271/2, Prague 8}
\ead{hartman@iuuk.mff.cuni.cz}

\fntext[myfootnote]{© 2019. This manuscript version is made available under the CC-BY-NC-ND 4.0 license http://creativecommons.org/licenses/by-nc-nd/4.0/}
\begin{abstract}
We show that the independence number of a countably infinite connected HH-homogeneous graph that does not contain the Rado graph as a spanning subgraph is finite and present a classification of MB-homogeneous graphs up to bimorphism-equivalence as a consequence.
\end{abstract}


\end{frontmatter}

\section{Introduction}

The symmetry of graphs, or more generally relational structures, is usually measured by such numbers as the degree of transitivity or homogeneity of the natural action of their automorphism group. One of the strongest notions of symmetry is \emph{ultrahomogeneity}, defined as the property that any isomorphism between two finite induced subgraphs can be extended to an automorphism. This notion was generalized by Cameron and \Jarik, in \cite{CameronNesetril:2006}, to {\em homomorphism-homogeneity}, requiring that any local homomorphism (that is, a homomorphism between finite induced substructures) extends to an endomorphism of the ambient structure. By specifying the type of local homomorphism and endomorphism, several new \emph{morphism-extension classes} were introduced by Lockett and Truss (see \cite{LockettTruss:2014}), each denoted by a pair of characters as XY and defined by the condition that any local $X$-morphism extends to a global $Y$-morphism. Here ${\rm X\in\{H,M,I\}}$ stands for \emph{homo, mono} or \emph{iso} and ${\rm Y\in\{H,A,I,B,E,M\}}$ stands for \emph{homo, auto, iso, bi, epi} or \emph{mono}. Thus, for example, the notion of homomorphism-homogeneity above is what we will call HH-homogeneity, and ultrahomogeneity is IA-homogeneity. In this paper, we will focus on the class of MB-homogeneous graphs, where any local monomorphism is a restriction of a bijective endomorphism (\emph{bimorphism}) of the ambient graph.

One of the main tasks in this area is classification, i.e., determining, given a language $L$ and a set of axioms $T$, all countable $L$-structures satisfying $T$ that fall into individual morphism-extension classes. A classic example of successful classification is the Lachlan-Woodrow theorem~\cite{LachlanWoodrow:1980}, which in our notation is a classification of IA-homogeneous graphs. Classification theorems can have broad implications because some of the classes appear in other areas of mathematics. For example, IH-homogeneous graphs appear in the area of graph limits~\cite{NesetrilPOM:2013} and HH-homogeneous structures appear as weakly oligomorphic structures that have found application in the research of infinite-domain CSPs~\cite{CMPech:2011}. We consider the extended family of classes as defined by Lockett and Truss consisting of 18 morphism-extension classes for general structures, which in the case of countable structures collapse to the fifteen presented in Figure~\ref{fig:lockettclasses}. 

\begin{figure}[h!]
\centering
\includegraphics[scale=0.8]{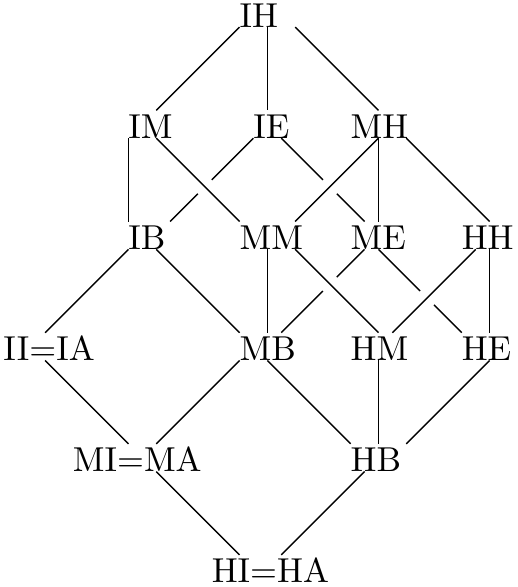}
\caption{Morphism-extension classes of countable structures, partially ordered by $\subseteq$. See~\cite{LockettTruss:2014} for more details.}
\label{fig:lockettclasses}
\end{figure}

A first approach to the problem of classifying countable homomorphism-homogeneous $L$-structures satisfying the axioms in $T$ is to determine the partial order of morphism-extension classes of such $L$-structures, because some of the classes in Figure \ref{fig:lockettclasses} may turn out to be equal for the $L$-structures in question. In the case of graphs, progress even in this simpler question has been slow: in the original 2006 paper, Cameron and \Jarik asked whether the classes HH and MH are equal for countable graphs, and the equality of these classes was established in 2010 for countable graphs by Rusinov and Schweitzer~\cite{RusinovSchweitzer:2010}. It was only last year that all the equalities and inequalities between morphism-extension classes of countable graphs were settled. We now know the partial order of morphism-extension classes of graphs. For related results regarding the equality of MH and HH in binary structures, we refer the reader to~\cite{Hartman:2014} and~\cite{ArandaHartman:2018}. 

Among the first results in the area is the fact that any countable graph that contains the Rado graph $\mathcal R$ as a spanning subgraph is HH-homogeneous. Apart from these and some trivial cases like disjoint unions of complete graphs, no countable graphs were known to exist in the class HH at the end of \cite{CameronNesetril:2006}, prompting the authors of the original paper to ask for examples of countably infinite, connected, HH-homogeneous graphs that do not contain $\mathcal R$ as a spanning subgraph.

The first examples of such graphs were presented by Rusinov and Schweitzer in \cite{RusinovSchweitzer:2010}, but up to the date of this writing, a full classification of countable HH-homogeneous graphs, or, more ambitiously, IH-homogeneous graphs (by which we mean a ``reasonable,'' list of the contents of each morphism class) does not exist. We hasten to mention here that the qualifier ``reasonable'' is important and in general means ``up to the appropriate equivalence relation, depending on the class.'' 

Most of the morphism-extension classes of countable graphs are uncountable (the exceptions are the countably infinite IA and the finite MI=MA and HM=HI=HB=HA), but uncountability by itself does not preclude a class from being classifiable. Cherlin successfully classified the uncountably many ultrahomogeneous directed graphs in \cite{cherlin1998classification} into countably many classes, one of which is uncountable. This is an example of the type of classification to which we aspire: only countably many classes, and few of them, preferably only one, uncountable.

In classes where a \Fraisse theorem is known, the uniqueness conditions for the limit are the gold standard for classification. This does not mean, however, that a given class is actually classifiable up to that equivalence relation. For ultrahomogeneous structures, limits are unique up to isomorphism, but, as we have seen, the classification of ultrahomogeneous directed graphs includes a class with uncountably many pairwise non-isomorphic structures. 

Most of the \Fraisse-type theorems for morphism-extension classes XY were found by Coleman \cite{Coleman:2018}, improving previous results of C. Pech and M. Pech \cite{Pech2016towards}, though no \Fraisse-type theorem is known for some morphism-extension classes.

For MB-homogeneous structures, the equivalence relation coming from the \Fraisse-type theorem is B-equivalence, which holds when every partial isomorphism with finite domain between two MB-homogeneous $L$-structures of the same age can be extended to a bimorphism. In the case of MB-homogeneous graphs, isomorphism or even the weaker notion of B-equivalence have been shown to produce uncountably many equivalence classes. By Theorem 3.20 from \cite{ColemanEvansGray:2019}, there exist $2^{\aleph_0}$ countably infinite, pairwise non-B-equivalent graphs in the bimorphism-equivalence class of the Rado graph. This rules out a classification up to B-equivalence. The next best candidate for equivalence relation in a classification of MB-homogeneous graphs is bimorphism-equivalence. 

The main contribution of this paper is Theorem \ref{thm:mainthm}, an analogue of the Lachlan-Woodrow theorem for MB-homogeneous graphs, providing a positive answer to Problem 4.12 from \cite{ColemanEvansGray:2019} (i.e., a classification of MB-homogeneous graphs up to bimorphism-equivalence). Theorem \ref{thm:mainthm} arises as a consequence of a preliminary result in our efforts towards a classification of HH-homogeneous graphs. Along with this result, we give a bound on the independence number of a countable connected HH-homogeneous graph that does not contain $\mathcal R$ as a spanning subgraph, in terms of the highest independence number of the neighbourhood of a vertex. In short, HH-homogeneity implies that in such a graph $G$ there is a finite upper bound $\oo{G}$ on the independence number of induced subgraphs defined as $N(v)$, and the independence number of $G$ cannot be ``too large'' compared to $\oo{G}$. 

The paper is organized as follows: first we prove in Section \ref{sec:alphainf} that HH-homogeneous graphs that do not contain the Rado graph as a spanning subgraph have finite independence number, and then we use that result in Section \ref{sec:MBclass} to classify MB-homogeneous graphs.

\section{The independence number of countable HH-homogeneous graphs}\label{sec:alphainf}
We start by fixing notation and recalling or introducing a few definitions. 
\begin{definition}~
\begin{enumerate}
\item{A \emph{graph} is a set equipped with a binary irreflexive symmetric relation. As such, a graph is a pair $G=(V,E)$, where $V$ is a set and $E\subset V^2$ satisfies $(x,y)\in E\rightarrow (y,x)\in E$ and for all $x\in V$ the pair $(x,x)$ is not in $E$. The elements of $V$ are called \emph{vertices}, the elements of $E$ are \emph{edges}, and the pairs of distinct vertices that do not satisfy $E$ are called \emph{nonedges}. Most of the time we will not distinguish between a graph and its vertex set. We will use the symbol $\sim$ to denote the edge relation.}
\item{For any set $X$, $\fin{X}$ denotes the set of finite subsets of $X$, and ${X}\choose{k}$ denotes the set of subsets of $X$ of size $k$. Abusing notation, when $G$ is a graph we always think of $A\in\fin{G}$ as the finite subgraph induced by $A$ in $G$ (definition below, item \ref{induced}).}
\item{In a graph $G$, a vertex $u$ is a \emph{neighbour} of $v$ if $u\sim v$ holds. The \emph{neighbourhood} of $v\in G$ is the set $\{w\in G:w\sim v\}$. For $S\in\fin{G}$, $N(S)$ is the set $\{v\in G:\forall s\in S(v\sim s)\}$. We call $N(S)$ the \emph{common neighbourhood} of $S$. When $S$ is a singleton, we write $N(v)$ for $N(\{v\})$.}
\item{The \emph{degree} of a vertex $v$ in $G$ is the cardinality of $N(v)$; its \emph{codegree} is the cardinality of $G\setminus(N(v)\cup\{v\})$. }
\item{A set of vertices $I\subset G$ is \emph{independent} if for all $u,v\in I$ we have $u\not\sim v$. The \emph{independence number} of a graph $G$, denoted by $\alpha(G)$, is defined as $\sup\{|X|:X\text{ is an independent subset of }G\}$.}
\end{enumerate}
Let $G=(V,E)$ and $H=(V',E')$ be graphs.
\begin{enumerate}
\setcounter{enumi}{5}
\item{The \emph{complement} of $G$ is the graph $\overline{G}=(V,V^2\setminus(E\cup\{(v,v):v\in V\}))$.}
\item{A \emph{homomorphism} from $G$ to $H$ is a function $h:V\to V'$ that preserves edges: if $x\sim y$ in $G$, then $h(x)\sim h(y)$ in $H$. A \emph{monomorphism} is an injective homomorphism and a \emph{bimorphism} is a bijective homomorphism.}
\item{$H$ is a \emph{subgraph} of $G$ if $V'\subseteq V$ and $E'\subseteq E$. If $V'=V$, then $H$ is a \emph{spanning} subgraph of $G$. If $E'=E\cap(V')^2$, then $H$ is an \emph{induced} subgraph of $G$. }\label{induced}
\item{$\age{G}$ is the set of isomorphism types of finite induced subgraphs of $G$. We think of the elements of $\age{G}$ as graphs on disjoint vertex sets that do not contain vertices from $G$, and we view the induced subgraphs of $G$ as images of elements of $\age{G}$ under embeddings.}
\item{Let $X\in\fin{G}$. If $c\in N(X)$, we call $c$ a \emph{cone} over $X$. If $d\not\sim x$ for all $x\in X$ and $d\notin X$, then we refer to $d$ as a \emph{co-cone} over $X$. Note that $d$ is a co-cone over $X$ if and only if $d$ is a cone over $X$ in $\overline{G}$.}
\item{A graph $G$ has property \trg if every finite induced subgraph of $G$ has a cone in $G$.}
\item{A graph $G$ has property \ctrg if every finite induced subgraph of $G$ has a co-cone in $G$.}
\end{enumerate}
\end{definition}

The Rado graph $\mathcal R$ is the unique countable graph with the following property: for all disjoint $A,B\in\fin{\mathcal R}$, there exists a vertex $x\in\mathcal R$ such that $x$ is a cone over $A$ and a co-cone over $B$. $\mathcal R$ is ultrahomogeneous and its age is the set of isomorphism types of finite graphs.

It is easy to prove that a countably infinite graph satisfies \trg iff it contains the Rado graph as a spanning subgraph. Property \ctrg holds for $G$ exactly when $\overline G$ has \trg.

\begin{definition}
Let $G$ be an infinite graph.
\begin{enumerate}
\item{Define $\kk{G}$ as the subset of $\age{G}$ consisting of all $A$ for which there exists an embedding $e:A\to G$ such that $G$ contains a cone over $e[A]$}
\item{Define $\okk G$ as the subset of $\age{G}$ consisting of all $A\in\age{G}$ for which there exists an embedding $e:A\to G$ such that no vertex in $G\setminus e[A]$ is a cone in $G$ over $e[A]$}
\end{enumerate}
\end{definition}

In any graph $G$ it is true that $\age{G}=\kk{G}\cup\okk{G}$ because any $X\in\fin{G}$ either has a cone in $G$ or doesn't, but $\kk{G}$ and $\okk{G}$ are seldom disjoint. For example, in the two-way infinite path $P$, the element of the age isomorphic to a nonedge is in $\kk{P}\cap\okk{P}$ because it can be embedded in $P$ as a pair at distance 2 or a pair at a larger distance. The reader should not make the mistake of confusing $\okk{G}$ with the set of structures in $\age{G}$ with copies in $G$ that have a co-cone in $G$. 

\begin{proposition}\label{prop:nbhdhh}\label{prop:homneighs}
If $G$ is an HH-homogeneous graph and $S\in\fin{G}$, then the subgraph of $G$ induced by $N(S)$ is an HH-homogeneous graph.
\end{proposition}
\begin{proof}
Consider $A,B\in\fin{N(S)}$ and a homomorphism $f\colon A\to B$. Then we can extend $f$ to $f'\colon A\cup S\to B\cup S$ by fixing each $v\in S$, and the result is still a homomorphism. By HH-homogeneity, there exists an endomorphism $F\colon G\to G$ that extends $f'$. This $F$ maps $N(S)$ into $N(S)$, so $F|_{N(S)}$ is an endomorphism of $N(S)$.
\end{proof}

In particular, the neighbourhood set of any vertex is an HH-homogeneous graph. Homomorphism-homogeneity also implies some uniformity of degrees:
\begin{proposition}\label{prop:infdegree}
Let $G$ be a countably infinite HH-homogeneous graph. If $G$ contains a vertex of infinite degree, then all vertices have infinite degree.
\end{proposition}
\begin{proof}
Suppose that $v_0\in G$ has infinite degree. By Ramsey's theorem, one of the following two cases holds:
\begin{enumerate}
\item{$N(v_0)$ \emph{contains an infinite independent set $I$:} In this case, it follows that $G$ has property \trg. Indeed, consider any $A\in\fin{G}$ and $B\subset I$ with $|B|=|A|$. Any bijection $h:B\to A$ is a homomorphism because $I$ is independent, and so by HH-homogeneity there exists an endomorphism $H$ with $H|_B=h$. The image of $v_0$ under $H$ is a cone over $A$. Clearly, if $G$ satisfies \trg then every vertex has infinite degree.}
\item{$N(v_0)$ \emph{contains an infinite clique } $K$: Let $w$ be any vertex of $G$. The mapping $v_0\mapsto w$ is a homomorphism, so by HH-homogeneity it extends to an endomorphism $H$ of $G$. By definition, the restriction of an endomorphism to a clique is injective, and therefore $N(w)$ contains an infinite clique and in particular $w$ has infinite degree in $G$.}
\end{enumerate}
\end{proof}

In a partial order $(P,\leq)$, we call $X\subset P$ \emph{downward closed} if given any $x\in X$ and $y\in P$, $y\leq x$ implies $y\in X$. Similarly, $Y\subset P$ is \emph{upward closed} if for all $y\in Y$ and all $x\in P$, $x\geq y$ implies $x\in Y$. Naturally, we consider the empty set to be upward and downward closed.

Throughout the paper, we will prove the possibility of extending a homomorphism with finite domain $f$ to an endomorphism by proving that $f$ can be extended to a homomorphism $f'$ whose domain is the domain of $f$ plus any new vertex; in fact we may assume that the new vertex is a cone over the domain of $f$, as finding an image for a cone implies finding an image for any vertex. This is known as the \emph{one-point extension property} (or \emph{weak one-point extension property} if we assume the new vertex is a cone), and it is known that homomorphism-homogeneity is equivalent to either of these properties for finite and countably infinite structures, see \cite{pech2009local}, particularly section 2.3.

Write $A\preceq B$ if there exists a surjective homomorphism $A\to B$. This relation is a partial order on $\age{G}$. Observe that in this order $A\preceq B$ implies $|A|\geq|B|$. The next two propositions characterise countable HH-homogeneous graphs with vertices of infinite degree. The first one is the one-point extension property adjusted to our language; the second statement is almost-known in the sense that the proof appeared in \cite{RusinovSchweitzer:2010} for connected HH-homogeneous graphs but in fact only requires the weaker hypothesis of infinite degree for all vertices. We include the proofs because the results are crucial for subsequent arguments.

\begin{proposition}\label{prop:3Conditions}
Let $G$ be a countable graph. Then $G$ is HH-homogeneous iff the following two conditions hold.
\begin{enumerate}
\item{$\kk{G}\cap\okk{G}=\varnothing$ and}\label{DisjCond}
\item{$\kk{G}$ upward-closed in $(\age{G},\preceq)$ (equivalently, $\okk{G}$ downward-closed in $(\age{G},\preceq)$
)}\label{kkCond}\label{okkCond}
\end{enumerate}
\end{proposition}
\begin{proof}
Suppose first that $G$ is HH-homogeneous. 

The first condition is clearly necessary because an isomorphism is a homomorphism: if $C\in\kk{G}\cap\okk{G}$, then we can find $X,Y\subset G$ isomorphic to $C$ and such that there is a cone $c$ over $X$ in $G$, but no vertex of $G$ is a cone over $Y$. An isomorphism $X\to Y$ is a homomorphism between finite substructures of $G$, but it cannot be extended to an endomorphism of $G$ since there is no possible image for $c$. 

The first condition allows us to abuse notation and simply say $D\in\kk{G}$ (or $D\in\okk{G}$) for a finite $D\subset G$, without mentioning the embedding, if $G$ is HH-homogeneous. Formally, $D\in\kk{G}$ means that there is an copy of $D$ in $G$ such that $G$ contains a cone over $D$, but as we have seen, this is true of \emph{any} copy of $D$ in $G$, or, equivalently, true for every embedding of the element of the age isomorphic to $D$ into $G$.

Next, we prove the equivalence of the two conditions in item \ref{okkCond} if condition \ref{DisjCond} holds. As noted before, $\kk{G}\cup\okk{G}=\age{G}$ for any graph, so condition \ref{DisjCond} says that $\kk{G}$ and $\okk{G}$ form a partition of $\age{G}$. If $\kk{G}$ is upward-closed in $(\age{G},\preceq)$ and $A\in\okk{G}$, then for any $B\preceq A$ it must be the case that $B\in\okk{G}$, as otherwise we obtain $B\in\kk{G}$ and by upward-closedness $A\in\kk{G}$, contradicting $\kk{G}\cap\okk{G}=\varnothing$. The other direction follows from a similar argument.


If $G$ is HH-homogeneous and $A\in\fin{G}$ has a cone $c\in G$, then for any surjective homomorphism $h:A\to B$, an extension $H$ will necessarily map $c$ to a cone over $B$, so condition \ref{kkCond} is satisfied.

For the converse, suppose that conditions \ref{DisjCond} and \ref{kkCond} hold and let $f:A\to B$ be a surjective homomorphism between finite induced subgraphs of $G$. We will show that $f$ can be extended to any superset $A\cup\{a\}$ as a homomorphism. 

We need only to consider the cases $B\in\kk{G}$ and $B\in\okk{G}$. If $B\in\kk{G}$, then the new vertex $a$ can be mapped to any cone $c$ over $B$, and the resulting function is still a homomorphism. And if $B\in\okk{G}$, then $A\in\okk{G}$ by Condition \ref{okkCond} and given any $c\notin A$:
\begin{enumerate}
\item{If $c\not\sim v$ for all $v\in A$, it follows that $f\cup\{(c,d)\}$ is a homomorphism for any $d$.}
\item{If $N(c)\cap A\neq\varnothing$, then $c$ is a cone over $C\coloneqq N(c)\cap A$ and $f|_C$ is a surjective homomorphism from $C$ to its image, so $f[C]$ has a cone $d$ by condition \ref{kkCond}. Since $d\notin f[C]$ (which happens automatically because the edge relation implies inequality), $f'\coloneqq f\cup\{(c,d)\}$ is a homomorphism and extends $f$.}
\end{enumerate}
This concludes our proof.
\end{proof}

We will use $\ndo{G}$ to denote the endomorphism monoid of $G$.

\begin{proposition}\label{prop:clquecond}
If $G$ is a countable HH-homogeneous graph with vertices of infinite degree, then for every $C\in\kk{G}$ there exists an infinite clique $K\subset G$ such that every vertex of $K$ is a cone over $C$.\label{CliqueCond}
\end{proposition}
\begin{proof}
Take $C\in\fin{G}$ and suppose that $c\in G$ is a cone over $C$. Since the degree of $c$ is infinite, there exists $u\in N(c)\setminus C$. The mapping $h\colon C\cup\{u\}\to C\cup\{c\}$ given by 
\[
h(v)=\begin{cases}
v&\text{if } v\in C \\
c&\text{if } v=u
\end{cases}
\]
is a homomorphism, so there is $H\in\ndo{G}$ that extends it. The image of $c$ under $H$ is a cone over $C\cup\{c\}$. This argument proves that the set of cones over $C$ is an infinite HH-homogeneous graph (Proposition \ref{prop:homneighs}) that contains no finite $\subseteq$-maximal cliques. Since every clique is contained in a maximal one, there exists an infinite clique consisting of cones over $C$.
\end{proof}

The complete bipartite graph $K_{1,n}$ is sometimes called an $n$-star. The following result is part of Proposition 2.1 in \cite{CameronNesetril:2006}.
\begin{proposition}\label{prop:alphabound}
Suppose that $G$ is an HH-homogeneous graph. If $G$ does not satisfy \trg, then there is a finite $N$ such that $G$ does not embed $n$-stars for any $n>N$. 
\end{proposition}
\begin{proof}
If $G$ embeds arbitrarily large $n$-stars, then we can follow the proof of Case 1 in Proposition \ref{prop:infdegree} to show that $G$ satisfies \trg.
\end{proof}

\begin{definition}
Let $G$ be a graph. Define the \emph{star number} of $G$ as $$\oo{G}\coloneqq\sup\{\alpha(N(v)):v\in G\}$$ if the supremum is finite, and as $\infty$ otherwise.
\end{definition}

We can rephrase Proposition \ref{prop:alphabound} above as saying that if $G$ is HH-homogeneous with \ntrg, then $\oo{G}$ is finite. In particular, there exists a finite $n$ such that each $N(v)$ is a $\overline{K_n}$-free HH-homogeneous graph (Proposition \ref{prop:homneighs}).

\begin{definition}\label{def:coord}
With the notions from above:
\begin{enumerate}
\item{If $D,X$ are subsets of $G$, we say that $D$ \emph{dominates} $X$ if $$X\subseteq\bigcup\{N(d):d\in D\}.$$ A set $D$ is a \emph{dominating set of }$G$ if $D$ dominates $G\setminus D$.}
\item{Let $G$ be a graph with finite star number $\oo{G}\geq 1$ and $I\subseteq G$. We say that $I$ is a \emph{directory} of $G$ if $I$ is an independent dominating set of $G$ and 
\begin{enumerate}
\item{$|I|=\alpha(G)$ if $\alpha(G)$ is finite, or}
\item{$|I|\geq2\oo{G}-1$ if $\alpha(G)$ is infinite.}
\end{enumerate}}
\item{Let $I$ be a directory of $G$. The \emph{address of $x\in G$ with respect to $I$}, denoted by $\rm{address}_I(x)$ or simply $\coord x$, is given by $$\rm{address}_I(x)=\begin{cases} N(x)\cap I&\mbox{ if } x\notin I\\ \{x\}&\mbox{ otherwise.}\end{cases}$$ We will write $\rm{address}_I(A)$ instead of $\bigcup\{\rm{address}_I(a):a\in A\}$.}
\end{enumerate} 
\end{definition}

Every maximal independent set in $G$ dominates $G$, but not every maximal independent set is a directory. The two notions are not very far from one another, however. There are only two ways in which a maximal independent subset $I$ of a graph $G$ with finite star number fails to be a directory, namely:
\begin{enumerate}
\item{$I$ contains fewer than $\alpha(G)$ elements and $\alpha(G)$ is finite, or}
\item{$I$ contains fewer than $2\sigma(G)-1$ elements and $\alpha(G)$ is infinite.}
\end{enumerate}

We also remark that not all maximal independent subsets of an HH-homo\-geneous graph have the same size. The difference can be seen in the original examples of connected HH-homogeneous graphs with \ntrg from \cite{RusinovSchweitzer:2010}, which we reproduce below. 

\begin{example}\label{ex:rs}
Given $n\geq 3$, take an independent set $A_n=\{a_0,\ldots,a_{n-1}\}$ and an infinite clique disjoint from $A_n$, partitioned into $n$ infinite subsets $C_0,\ldots, C_{n-1}$. Finally, add all edges of the form $\{c,a_j\}$ where $c\in\bigcup\{C_i:i\neq j\}$, and call the resulting graph $RS(n)$. It is clear from the construction that the finite subgraph $A_n$ does not have a cone in $RS(n)$, and HH-homogeneity is easy to verify using Proposition \ref{prop:3Conditions}. In $RS(n)$, each pair $\{c,a_i\}$ with $c\in C_i$ is a maximal independent set, but $A_n$ is the only directory of $RS(n)$.
\end{example}

In an HH-homogeneous graph with infinite independence number, we cannot simply require the size of the directory to be $\alpha(G)$ because HH-homogeneity does not guarantee the existence of an infinite independent set in $G$ when $\alpha(G)$ is infinite. All we know in that case is that $G$ embeds arbitrarily large finite independent sets. Our next example illustrates this point.

\begin{example}
Let $G$ be the complement of the disjoint union of $\{K_n:n\in\omega\}$. Then $G$ satisfies \trg and is therefore HH-homogeneous, has infinite independence number, and does not embed an infinite independent set.
\end{example}

\begin{notation}
Let $G$ be a graph with finite star number. If $I$ is a directory of $G$ and $S\subseteq I$, then $K_S\coloneqq\{v\in G:N(v)\cap I=S\}$. We call $K_S$ the \emph{exact neighbourhood} of $S$ (with respect to $I$). 
\end{notation}

\begin{lemma}\label{lem:nsks}
Let $G$ be a graph with finite star number. If $I$ is a directory of $G$, then for any $S\in{I\choose{\oo{G}}}$, $K_S=N(S)$.
\end{lemma}
\begin{proof}
It is clear that $K_S\subseteq N(S)$ even when $|S|<\oo{G}$. Consider $v\in N(S)$; if $v\in N(S)\setminus K_S$, then $S\subsetneq N(v)\cap I$, and so $\alpha(N(v))>\oo{G}$, impossible. It follows that $v\in K_S$ and the exact and common neighbourhoods of $S$ are equal.
\end{proof}

\begin{proposition}\label{prop:DisjNoEdges}
Let $G$ be a graph with finite star number, and suppose that $I$ is a directory of $G$. Then for any disjoint $S,T\in{I\choose\oo{G}}$, there are no edges $v\sim w$ with $v\in K_S$ and $w\in K_T$.
\end{proposition}
\begin{proof}
Suppose that $S,T\in{I\choose\oo{G}}$ are disjoint, an let $v\in K_S$, $w\in K_T$. If $v\sim w$, then $\{w\}\cup\coord v$ would be an independent set of size $\oo{G}+1$ contained in $N(v)$, impossible by the definition of $\oo{G}$ (see Figure \ref{fig:fig2}).
\begin{figure}[h!]
\centering
\includegraphics{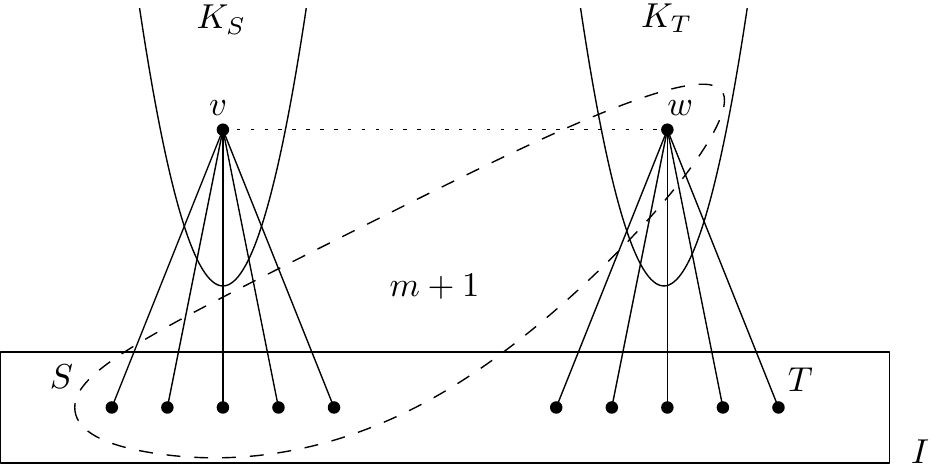}
\caption{Proposition \ref{prop:DisjNoEdges}. If the dotted line were an edge, then the dashed set would be an independent subset of size $\oo{G}+1$ in the neighbourhood of $v$.}\label{fig:fig2}
\end{figure}
\end{proof}

\begin{lemma}\label{lem:HartmansTrick}
Let $G$ be a countably infinite HH-homogeneous graph with vertices of infinite degree and finite star number satisfying $\alpha(G)\geq 2\oo{G}-1$. Suppose that $I$ is a directory of $G$. Then for all non-disjoint $S,T\in{I\choose\oo{G}}$ and all $v\in K_S$, the set $N(v)\cap K_T$ is infinite.
\end{lemma}
\begin{proof}
If $S=T$, then $K_S$ is an infinite HH-homogeneous graph (Propositions \ref{prop:homneighs} and \ref{prop:clquecond}) with finite independence number (bound by the star number of $G$), so all its vertices have infinite degree within $K_S$ (Propositions \ref{prop:infdegree} and \ref{prop:clquecond}), and thus the result holds in this case. We may now assume $S\neq T$.

Write $m$ for $\oo{G}$. The proof consists of two pairs of claims of increasing strength. We start with the simplest case:
\begin{claim}\label{clm:manyedges1}
If $S,T\in{I\choose m}$ satisfy $|S\cap T|=1$, then every $v\in K_T$ has a neighbour in $K_S$. 
\end{claim}
\begin{proof}
Let $S=\{c_1,\ldots,c_m\}$ and $T=\{c_m,\ldots,c_{2m-1}\}$. Take $u\in K_S$ and $v\in K_T$, and consider the set $D=\{c_1,\ldots,c_{m-1},v\}$. Since $c_m$ is the only element common to $S$ and $T$, $D$ is an independent set of size $m$. By HH-homogeneity, $D$ has a cone $z$. Observe that $z\notin I$ because it has edges to elements of $I$. Observe also that if $N(z)\cap T=\varnothing$, then $\{z,c_m,\ldots,c_{2m-1}\}$ would be an independent set of size $m+1$ contained in $N(v)$, which is a contradiction, so there must be exactly one edge with one endpoint $z$ and another $d\in T$ (see Figure \ref{fig:fig3}). An endomorphism of $G$ extending the homomorphism $f:\{c_1,\ldots,c_{m-1},d,v\}\to S\cup\{v\}$ that fixes $c_1,\ldots,c_{m-1},v$ and maps $d\mapsto c_m$ sends $z$ to a neighbour of $v$ in $K_S$.
\begin{figure}[h!]
\centering
\includegraphics{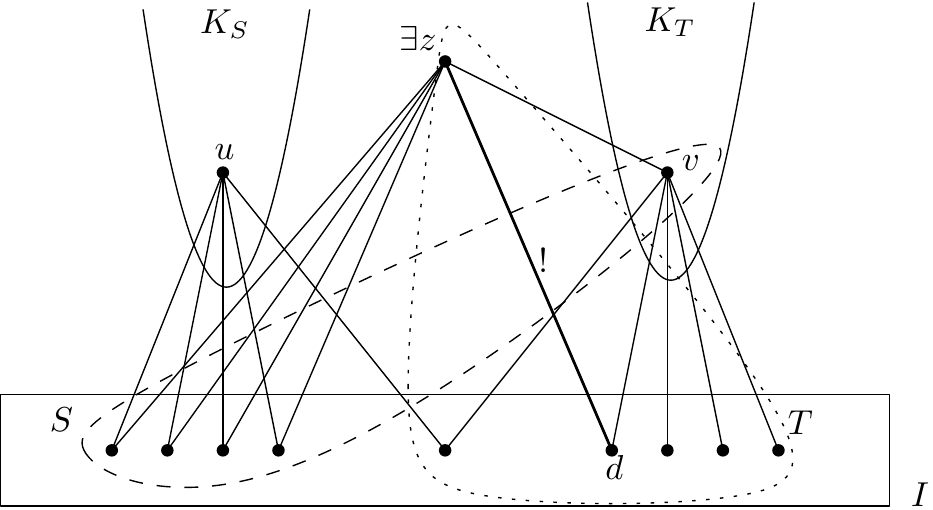}
\caption{Claim \ref{clm:manyedges1}. By HH-homogeneity, the dashed set has a cone $z$. If the heavier edge were not there, then the dotted set would be a large independent set in $N(v)$, so exactly one such edge is forced.}\label{fig:fig3}
\end{figure}
\end{proof}
\begin{claim}
If $S,T\in{I\choose m}$ satisfy $S\cap T\neq\varnothing$, then every $v\in K_T$ has a neighbour in $K_S$. 
\end{claim}
\begin{proof}
Claim \ref{clm:manyedges1} is the first step of an inductive argument on $|S\cap T|$ up to $|S\cap T|=m-2$. 

Let $1\leq k\leq m-3$, such that every $v\in K_T$ has a neighbour in $K_S$ if $|S\cap T|=k$. Let $S,T\in{{I}\choose m}$ with $|S\cap T|=k+1$. Pick a vertex $x\in S\cap T$. Because $k+1\geq 2$ and $\alpha(G)\geq2\oo{G}-1$, there exists $w\in I\setminus(S\cup T)$. Now the independent set $S'=(S\cup\{w\})\setminus\{x\}$ has $m$ elements and $|S'\cap T|=k$, and so $K_{S'}$ contains a neighbour $u'$ of $v$, by the induction hypothesis. Define $f:S'\cup\{v\}\to S\cup\{v\}$ as $$f(s)=\begin{cases}s&\mbox{ if } s\in(S'\cup\{v\})\setminus\{w\}\\
x&\mbox { if } s=w.
\end{cases}$$

This function is a homomorphism, so by HH-homogeneity it has an extension $F\in\ndo{G}$. The image of $S'$ under $F$ is $S$, so the cone $u'$ over $S'$ is mapped by $F$ to a cone over $S$ (that is, an element of $K_S$), which is a neighbour of $v$.

\end{proof}

The preceding two claims, together with Proposition \ref{prop:DisjNoEdges}, show that in a connected HH-homogeneous graph $G$ with finite star number, if $I$ is a directory of $G$, then the structure of the set of vertices with $\oo{G}$ neighbours in $I$ is closely related to that of the intersection graph of ${I\choose{\oo{G}}}$. Indeed, if we define a graph whose vertices are the exact neighbourhoods of the elements of ${I\choose{\oo{G}}}$, with an edge $K_S\sim K_T$ if there exist $v\in K_S$ and $w\in K_T$ such that $v\sim w$ in $G$, then we obtain the intersection graph of ${I\choose{\oo{G}}}$.  

\begin{claim}
If $S,T\in{I\choose m}$ satisfy $S\cap T\neq\varnothing$, then all sets of the form $N(u)\cap K_T$ with $u\in K_S$ are finite or all are infinite.
\end{claim}
\begin{proof}
For any $u,w\in K_S$, the map $h$ fixing $S\cup T$ and sending $w\mapsto u$ is a homomorphism, and any global extension $H$ of $h$ will map $N(w)\cap K_T$ to a subset of $N(u)\cap K_T$. If $N(u)\cap K_T$ is finite and $N(w)\cap K_T$ is infinite then some vertex in $N(u)\cap K_T$ has infinite preimage under $H$ in $N(w)\cap K_T$. The preimage of a vertex under a homomorphism is an independent set, and therefore $\alpha(N(w))$ is infinite, contradicting the finiteness of $\oo{G}$.
\end{proof}

\begin{claim}
If $S,T\in{I\choose m}$ satisfy $S\cap T\neq\varnothing$, then $N(u)\cap K_T$ is infinite for all $u\in K_S$.
\end{claim}
\begin{proof}
Suppose for a contradiction that $N(u)\cap K_T$ is the finite set $A$. Then $u$ is a cone over $S\cup A$, and by Proposition \ref{CliqueCond} there exist infinitely many cones over this set. But being a cone over $S$ is equivalent to being in $K_S$, and so there is an infinite clique $C$ contained in $K_S$ such that for any $a\in A$ and $c\in C$ we have $c\sim a$. 

Now we prove that this is not possible in an HH-homogeneous graph with finite star number. Let $w$ be any element of $A$ and $h\colon S\cup T\cup\{u,w\}\to S\cup T\cup\{u,w\}$ be any bijection that fixes $S\cap T$, maps $T\setminus S$ to $S\setminus T$, sends $w$ to $u$ and vice versa. Then  $h$ is an homomorphism, and an endomorphism $H$ that extends $h$ will map $N(w)\cap K_S$ into $N(u)\cap K_T$. This contradicts our hypothesis because the infinite clique $C$ needs to be mapped by $H$ into the finite graph $A$.

\end{proof}
This concludes the proof of Lemma \ref{lem:HartmansTrick}.
\end{proof}

\begin{definition}
Let $G$ be a graph and $I$ be a directory of $G$. We define the \emph{domination number of $S\in\fin{G}$ over }$I$ (or its $I$-\emph{domination number}) as the value of the function $d_I\colon\fin{G}\to\N$ given by 
\[
d_I(S)=\begin{cases}\min\{|A|:A\subset I\text{ and }A\text{ dominates }S\}&\mbox{ if } S\cap I=\varnothing\\|S\cap I|+d_I(S\setminus B_S)&\mbox{ otherwise.}\end{cases}
\]
In the second case, $B_S=\{x:\exists s\in S\cap I(x\in N(s))\}\cup(S\cap I)$.
\end{definition}

\begin{lemma}\label{lem:WideCliques}
Let $G$ be a countably infinite HH-homogeneous graph with vertices of infinite degree, $\oo{G}\geq 2$, and $\alpha(G)\geq 2\oo{G}-1$, and let $I$ be a directory of $G$. Then there exist copies of $K_3$ in $G$ with $I$-domination number 2.
\end{lemma}
\begin{proof}
Choose and fix $S,T,U\in{I\choose\oo{G}}$ with $|S\cap T|=\left\lfloor\frac{\oo{G}}{2}\right\rfloor$ and $U\subseteq S\triangle T$, so that $S\cap T\cap U=\varnothing$ and $|S\cup T\cup U|= \oo{G}+\left\lceil\frac{\oo{G}}{2}\right\rceil$. We know from Lemma \ref{lem:HartmansTrick} that any $v\in K_S$ has infinitely many neighbours in $K_T$ and $K_U$, so fix $v\in K_s$, and let $w\in K_T$ and $z\in K_U$ be neighbours of $v$. Suppose that $w\sim z$. Then $v,w,z$ form a copy $C$ of $K_3$. We claim that $d_I(C)=2.$ This follows from two facts: first, any finite $Y\subset G$ is dominated by some $X\subseteq\coord{Y}$; and second, if a single vertex $c\in I$ dominates $Y$, then $c$ is by definition contained in $\bigcap\{N(y)\cap I:y\in Y\}$. But since $S\cap T\cap U=\varnothing$ and these sets are the addresses of the vertices in $C$, we know that the $I$-domination number of $C$ is at least 2. Clearly, the set containing one vertex from $S\cap T$ and one from $T\cap U$ dominates $C$. 

Suppose now that $w\not\sim z$. Let $w'\in K_T$ and $z'\in K_U$ form an edge, and define $f:S\cup T\cup U\cup\{w,z\}\to S\cup T\cup U\cup\{w',z'\}$ be the map fixing $S\cup T\cup U$ pointwise and sending $w\mapsto w', z\mapsto z'$. This is a homomorphism, and the image $v'$ of $v$ under a global extension forms a triangle with $w'$ and $z'$ and is in $K_S$ because $f$ fixes $S$ pointwise. Now the argument from the preceding paragraph proves that $\{w',z',v'\}$ is a copy of $K_3$ with $I$-domination number 2.
\end{proof}

\begin{lemma}\label{lem:Dominating}
Let $G$ be a graph with finite star number and let $I$ be a directory of $G$. If a finite set $X\subset G$ contains a vertex $x$ such that $|\coord{x}|=\oo{G}$, and $z\in G$ is a cone over $X$, then $\coord{z}\cap\coord{x}\neq\varnothing$. In particular, if $X$ consists of vertices with addresses of size $\oo{G}$, then $\coord{z}\cap\coord{X}$ dominates $X$.
\end{lemma}
\begin{proof}
The result is trivial if $z\in I$, so assume $z\notin I$. If $\coord{z}\cap\coord{x}$ were empty, then $\{z\}\cup\coord{x}$ would be an independent set of size $\oo{G}+1$ in $N(x)$. This is a contradiction, hence the first statement follows.

For the second assertion, suppose for a contradiction that $\coord{z}\cap\coord{X}$ does not dominate $X$. Then there exists $x\in X$ such that $x\notin\bigcup\{N(v):v\in\coord{z}\cap\coord{X}\}$, so $\{z\}\cup\coord{x}$ is an independent subset of size $\oo{G}+1$ in $N(x)$, impossible. 
\end{proof}

\begin{theorem}\label{thm:hhconntrg}
If $G$ is an countably infinite connected HH-homogeneous graph with finite star number $\oo{G}\geq 2$, then $\alpha(G)<2\oo{G}+\left\lceil\frac{\oo{G}}{2}\right\rceil-1$.
\end{theorem}
\begin{proof}

Suppose for a contradiction that $\alpha(G)\geq2\oo{G}+\left\lceil\frac{\oo{G}}{2}\right\rceil-1$. Then there is a directory $I$ with at least $2\oo{G}+\left\lceil\frac{\oo{G}}{2}\right\rceil-1$ vertices. We find $X\in\kk{G}, Y\in\okk{G}$ such that $X$ and $Y$ induce isomorphic subgraphs of $G$.

Take $S\in{{I}\choose{\oo{G}}}$ and a copy $C_1$ of $K_3$ in $K_S$ (we can find $C_1$ because $K_S=N(S)$ contains an infinite clique by Propositions \ref{prop:homneighs} and \ref{prop:clquecond}). This $C_1$ clearly has $I$-domination number 1, as witnessed by any $s\in S$. 

Select $T,U,V\in{{I}\choose{\oo{G}}}$ such that $T\cap V\cap U=\varnothing$, $|T\cap V|=\left\lfloor\frac{\oo{G}}{2}\right\rfloor$, and $U\subseteq V\triangle T$. Now, as in the proof of Lemma \ref{lem:WideCliques}, there is a copy $C_2$ of $K_3$ in $G$ with $I$-domination number 2 such that $T,U,V$ are the addresses of its vertices. Pick any set $D_1$ with $\oo{G}-1$ vertices from $I\setminus S$ and let $X$ be $C_1\cup D_1$. Define $Y$ as $C_2\cup D_2$, where $D_2$ is any subset of $I\setminus\coord{C_2}$ with $\oo{G}-1$ vertices. Then $X$ and $Y$ are isomorphic to the union of $K_3$ and an independent set of size $\oo{G}-1$.

\begin{figure}[h!]
\centering
\includegraphics{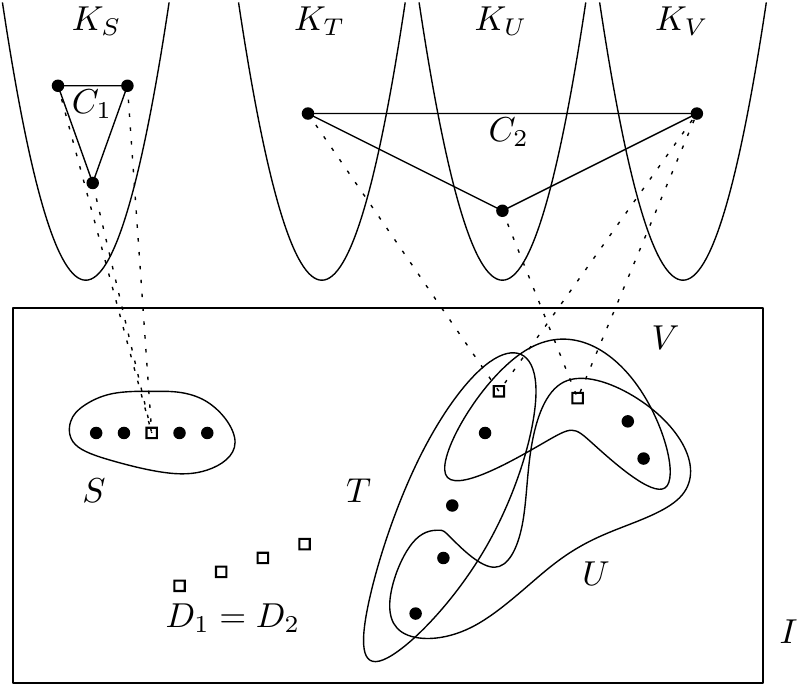}
\caption{The choice of addresses and triangles in the proof of Theorem \ref{thm:hhconntrg} with $\oo{G}=5$. To simplify the diagram, we assumed a $|I|\geq3\oo{G}+\left\lceil\frac{\oo{G}}{2}\right\rceil-1$, so that $D_1=D_2$ is possible. The square vertices in $S$ and $T\cup V$ dominate $C_1$ and $C_2$, respectively.}
\end{figure}

We claim that $X$ has a cone in $G$. To see this, consider the set $W=D_1\cup\{s\}$, where $s$ is any element of $S$. Let $z$ be any element of $K_W$. It is known that in a connected HH-homogeneous graph each vertex has infinite degree (Lemma 4 of \cite{RusinovSchweitzer:2010}), so we can apply Lemma \ref{lem:HartmansTrick} to conclude that $N(z)\cap K_S$ is infinite. Now Proposition \ref{prop:clquecond} and the finiteness of $\oo{G}$ imply that it contains an infinite clique, so in particular there is a copy of $K_3$, say $C_1'$ contained in $N(S\cup\{z\})$. Let $f\colon S\cup D_1\cup C_1'\to S\cup D_1\cup C_1$ be the homomorphism fixing $S\cup D_1$ pointwise and mapping $C_1'$ to $C_1$ bijectively. The image of $z$ under any extension of $f$ is a cone over $X$.

Next, we prove that $Y$ does not have a cone. By our choice of addresses, each vertex $v$ of $C_2$ satisfies $|\coord{v}|=\oo{G}$, so by Lemma \ref{lem:Dominating} the address of a cone $z$ over $C_2$ contains a dominating set $P_z$ for $C_2$, which in fact is $\coord{z}\cap\coord{C_2}$. By our choice of address and $D_2$, it is not possible for $Y$ to have a cone, as the neighbourhood of $z$ would contain $D_2\cup P_z$, an independent set of size at least $\oo{G}+1$.

The last two paragraphs contradict Condition \ref{DisjCond} in Proposition \ref{prop:3Conditions} and establish the Theorem.
\end{proof}

The bound on $\alpha(G)$ from Theorem \ref{thm:hhconntrg} above is tight. In $RS(3)$ (see example \ref{ex:rs}), we have $\oo{G}=2$ and $\alpha(G)=3=2\oo{G}+\left\lceil\frac{\oo{G}}{2}\right\rceil-2$.

\begin{corollary}\label{cor:hhconntrg}
If $G$ is an infinite connected HH-homogeneous graph with infinite independence number, then $G$ satisfies \trg.
\end{corollary}
\begin{proof}
If $G$ satisfies $\neg$\trg, then the star number of $G$ is finite, by Proposition \ref{prop:alphabound}. Hence, by Theorem \ref{thm:hhconntrg} $\alpha(G)$ is finite, too--- a contradiction.
\end{proof}

\section{MB-homogeneous graphs}\label{sec:MBclass}

In this section we use Corollary \ref{cor:hhconntrg} to classify MB-homogeneous graphs up to bimorphism-equivalence. We remind the reader that two relational structures $G$ and $H$ are bimorphism-equivalent if there exist bijective homomorphisms $F:G\to H$ and $J:H\to G$. ``Bimorphism-equivalent'' and ``isomorphic'' are distinct notions only for infinite structures. For graphs, bimorphism-equivalence means that $G$ is (isomorphic to) a spanning subgraph of $H$ and $H$ is (isomorphic to) a spanning subgraph of $G$.

The following theorem is an amalgamation of results from Cameron-\Jarik \cite{CameronNesetril:2006} and Coleman-Evans-Gray \cite{ColemanEvansGray:2019}:
\begin{theorem}\label{thm:bimRad}
A countably infinite graph with property \trg contains the Rado graph as a spanning subgraph. If in addition it satisfies \ctrg, then it is bimorphism-equivalent to the Rado graph.
\end{theorem}
An important fact from Coleman-Evans-Gray \cite{ColemanEvansGray:2019}:
\begin{theorem}\label{thm:MBcomp}
If $G$ is a MB-homogeneous graph, then its complement $\overline{G}$ is also MB-homogeneous.
\end{theorem}

\begin{remark}\label{rmk:MBHH}
Any MB-homogeneous graph is MH-homogeneous because a bimorphism is a homomorphism. It follows from the fact that MH=HH for graphs (see \cite{RusinovSchweitzer:2010}) that MB-homogeneous graphs are HH-homogeneous.
\end{remark}

Given two graphs $G$ and $H$ with disjoint vertex sets, we can form the \emph{graph composite} or \emph{lexicographic product} of $G$ and $H$, denoted by $G[H]$, as follows: the vertex set is $G\times H$ and $(g,h)\sim(g',h')$ if $g\sim g'$ in $G$ or $g=g'$ and $h\sim h'$ in $H$. In $G[H]$, each set of the form $\{g\}\times H$ induces an isomorphic copy of $H$ and for any function $f:G\to H$, the set $\{(g,f(g)):g\in G\}$ with its induced subgraph structure in $G[H]$ is isomorphic to $G$. We will use $I_\kappa$ to denote an independent set of size $\kappa$. The \emph{co-degree} of a vertex $v\in G$ is $|\{w\in G:w\neq v\wedge w\not\sim v\}|$.

\begin{corollary}\label{cor:InfDegCodeg}
If $G$ is not complete or null, countably infinite, and MB-homo\-geneous, then it is either connected or isomorphic to $I_\omega[K_\omega]$. If $G$ is not null, then every vertex has infinite degree and co-degree.
\end{corollary}
\begin{proof}
Since $G$ is not null, there is some connected component $C$ with at least one edge $u\sim v$; if $G$ is not connected, then there is $w\in G\setminus C$. If $G$ has only finitely many connected components, then the monomorphism $u\mapsto u$, $w\mapsto v$ cannot be extended to a surjective endomorphism. Therefore, $G$ is connected or has infinitely many connected components.

We know that $G$ is HH-homogeneous by Remark \ref{rmk:MBHH}, so if $G$ is disconnected, then each connected component is a clique. From these, only $I_\omega[K_\omega]$ is MB-homogeneous (Proposition 3.4 of \cite{ColemanEvansGray:2019}). 

The claims about degree and codegree are obviously true in $I_\omega[K_\omega]$. 

If $G$ is connected, then it follows from HH-homogeneity that every vertex has infinite degree (Proposition 1.1 (c) of \cite{CameronNesetril:2006}). 

Now suppose for a contradiction that $G$ is MB-homogeneous and some vertex $w$ has finite codegree $n>0$. We know by Theorem \ref{thm:MBcomp} that $\overline{G}$ is also MB-homogeneous, so it is an HH-homogeneous graph with a vertex of finite degree, so $\overline{G}$ cannot be connected. Disconnected HH-homogeneous graphs are unions of equipotent cliques, so $\overline G\cong I_\omega[K_{n+1}]$. Since this graph is not MB-homogeneous, this contradicts Theorem \ref{thm:MBcomp}.
\end{proof}

We now have enough information to classify MB-homogeneous graphs up to bimorphism equivalence. This answers a question from \cite{ColemanEvansGray:2019}.

\begin{theorem}\label{thm:mainthm}
Let $G$ be a countably infinite MB-homogeneous graph. Then $G$ is bi\-morphism-equivalent to one of the following or its complement:
\begin{enumerate}
\item{$K_\omega$,}
\item{$I_\omega[K_\omega]$,}
\item{The Rado graph $\mathcal R$.}
\end{enumerate}
\end{theorem}
\begin{proof}
If $G$ is a connected, countably infinite MB-homogeneous graph, then one of the following holds:
\begin{enumerate}
\item{$G\cong K_\omega,$}
\item{$G\cong\overline{I_\omega[K_\omega]}$,}
\item{$\overline{G}$ is connected.}
\end{enumerate}
The first two cases can be handled by Corollary \ref{cor:InfDegCodeg}. We turn our attention to the third one. Since $G$ is connected, it contains a copy of $K_\omega$, and for the same reason $\overline{G}$ contains a copy of $K_\omega$. It follows that $\alpha(G)=\alpha(\overline{G})=\omega$, and by Corollary \ref{cor:hhconntrg}, both $G$ and $\overline{G}$ satisfy \trg. We conclude that $G$ is bimorphism-equivalent to $\mathcal R$ (Theorem \ref{thm:bimRad}).
\end{proof}

\begin{remark}
In the first two cases of Theorem \ref{thm:mainthm}, the bimorphism is always an isomorphism. Only four of the uncountably many countable MB-homogeneous graphs are not bimorphism-equivalent to the Rado graph.
\end{remark}

\section{Acknowledgements}
The first author was funded by the ERC under the European Union's Horizon 2020 Research and Innovation Programme (grant agreement No. 681988, CSP-Infinity). The second author was partially supported by the ERC Synergy grant DYNASNET (grant agreement No. 810115).

We thank the anonymous referees for their careful reading of our work and their suggestions. We are particularly thankful to the referee who spotted a shortcut in the proof of Theorem \ref{thm:mainthm}
\bibliography{Morph}

\end{document}